\documentclass{article}
\usepackage{amsmath,amsthm,enumerate,amssymb}
\newtheorem{Th}{Theorem}
\newtheorem{Le}{Lemma}
\newtheorem{De}{Definition}
\newtheorem{Co}{Corollary}
\newtheorem{Rm}{Remark}

\newcommand{\e}{\mathrm{e}}
\newcommand{\ci}{\mathrm{i}}

\begin{document}
\title{Perturbations of Gibbs semigroups and the non-selfadjoint harmonic oscillator}
\author{Lyonell Boulton\footnote{Department of Mathematics and Maxwell Institute for Mathematical Sciences, Heriot-Watt University, Edinburgh, EH14 4AS, United Kingdom. E-mail: \texttt{L.Boulton@hw.ac.uk} }}

\date{2nd July 2018}

\maketitle

\begin{abstract}
Let $T$ be the generator of a $C_0$-semigroup $e^{-Tt}$ which is of trace class for all $t>0$ (a Gibbs semigroup). Let $A$ be another closed operator, $T$-bounded with $T$-bound equal to zero. In general $T+A$ might not be the generator of a Gibbs semigroup. In the first half of this paper we give sufficient conditions on $A$ so that $T+A$ is the generator of a Gibbs semigroup. We determine these conditions in terms of the convergence of the Dyson-Phillips expansion in suitable Schatten-von Neumann norms.  

In the second half of the paper we consider $T=H_\vartheta=-\e^{-\ci\vartheta}\partial_x^2+\e^{\ci\vartheta}x^2$, the non-selfadjoint harmonic oscillator, on $L^2(\mathbb{R})$ and $A=V$, a locally integrable potential growing like $|x|^{\alpha}$  at infinity for $0\leq \alpha<2$. We establish that the Dyson-Phillips expansion converges in $r$ Schatten-von Neumann norm in this case for $r$ large enough and show that $H_\vartheta+V$ is the generator of a Gibbs semigroup $\mathrm{e}^{-(H_\vartheta+V)\tau}$ for $|\arg{\tau}|\leq \frac{\pi}{2}-|\vartheta|\not=\frac{\pi}{2}$.  From this we determine high energy asymptotics for the eigenvalues and the resolvent norm of $H_\vartheta+V$.
\end{abstract}



\newpage

\section{Introduction}
The non-selfadjoint harmonic oscillator 
\[
      H_{\vartheta}=-\e^{-\ci\vartheta}\partial_x^2+\e^{\ci\vartheta}x^2
\qquad \qquad \Big(-\frac{\pi}{2}<\vartheta<\frac{\pi}{2}\Big),
\]
acting on $L^2(\mathbb{R})$ with domain $\operatorname{D}(H_{\vartheta})=H^2(\mathbb{R})\cap \widehat{H^2(\mathbb{R})}$ studied by Exner in \cite{1983Exner} and Davies in \cite{1999Davies2}, has become one of the reference models in the theory of pseudospectra and non-selfadjoint phenomena. \emph{Cf.} \cite{2017KrejcirikSiegl,2016Viola,2006Pravda-Starov,2001Boulton}, \cite[Chapter~14]{2007Davies} and \cite[p.105]{2005TrefethenEmbree}.  This operator is $J$-selfadjoint with respect to the conjugation $Ju(x)=\overline{u}(x)$ and $H^*_{\vartheta}=H_{-\vartheta}$, so it is selfadjoint only when $\vartheta=0$.  
As it is also m-sectorial, $H_\vartheta$ is the generator of a $C_0$-semigroup $\e^{-H_{\vartheta}\tau}$ for all $|\arg{\tau}|\leq \frac{\pi}{2}-|\vartheta|$.  In fact, the classical Mehler's formula extends to $\vartheta\not=0$ and non-real $\tau$ in a maximal angular semi-module which is much larger than this sector, rendering a trace class (Gibbs) semigroup. See \cite{2016Viola,2017AlemanViola} and Theorem~\ref{signandasymptoticsofrealpartswj} below.  

In this paper we consider perturbations of $H_\vartheta$ by locally integrable complex potentials $V$ such that
\begin{equation}     \label{precisecondpot}
      |V(x)|\leq a|x|^\alpha+b \qquad \qquad \forall x\in \mathbb{R}
\end{equation}
for some $0\leq \alpha <2$, $a>0$ and $b\in\mathbb{R}$. As $V$ is $H_{\vartheta}$-bounded with relative bound 0, the non-selfadjoint Schr{\"o}dinger operators $H_{\vartheta}+V$ are also $J$-selfadjoint in the same domain $\operatorname{D}(H_\vartheta+V)=\operatorname{D}(H_{\vartheta})$.

In Section~\ref{sec3} we show that $H_\vartheta+V$ is the generator of a Gibbs semigroup $\e^{-(H_\vartheta+V)\tau}$ for all $|\arg{\tau}|\leq \frac{\pi}{2}-|\vartheta|$ when $\vartheta\not=0$.  According to the work of Angelescu \emph{et al} \cite{1974Angelescuetal} and of Zagrebnov \cite{1989Zagrebnov}, see also \cite{2001CachiaZagrebnov}, a class~$\mathcal{P}$  perturbation\footnote{See \cite[p70]{1980Davies}.} of an m-sectorial operator whose real part is the generator of a Gibbs semigroup is also the generator of a Gibbs semigroup. But what is remarkable and not obvious for $H_\vartheta+V$ from this, is the fact that the trace class property extends all the way to the edges of the maximal sector. We obtain the latter, by showing that the Dyson-Phillips expansion of the perturbed semigroup converges in an $r$ Schatten-von Neumann norm for sufficiently large $r$ (it does not converge for $r\leq 2$ for $\alpha$ too close to 2). 

As the framework turns out to be general and may be applicable in other contexts, we begin by developing an abstract perturbation theory of generators of Gibbs semigroups in Section~\ref{sec1}. The results in the papers \cite{1974Angelescuetal} and \cite{1989Zagrebnov} mentioned above, rely on an inequality due to Ginibre and Gruber, \cite{1969GinibreGruber}, which cannot be easily extended to the non-sectorial setting. Therefore we take here a completely different route, that of the Dyson-Phillips expansion. This allows generators which are not necessarily m-sectorial, but the perturbations ought to be more than just of class~$\mathcal{P}$. They must  satisfy an analogous condition of integrability, but with respect to a Schatten-von Neumann norm. Details in Lemma~\ref{Th:classPCpperturbation} and Definition~\ref{De:classPCq} below.

The spectrum of $H_{\vartheta}$ is
$
     \operatorname{Spec}(H_{\vartheta})=\left\{2n+1\right\}_{n=0}^\infty,
$
with corresponding normalised eigenfunctions \cite{1999Davies}
\[
     \Psi_n(x)=\e^{\frac{\ci\vartheta}{4}}\Phi_n\left(\e^{\frac{\ci\vartheta}{2}}x\right)
\]
where $\{\Phi_n\}_{n=0}^\infty$ are the normalised eigenfunctions of $H_0$. In Section~\ref{sec3} (Corollary~\ref{Co:asymptoticeigenvalues}) we show that the eigenvalues of $H_\vartheta+V$ have a real part growing at least like $n$ and a distance from the rays $|\arg(z)|=|\vartheta|$ growing at least like $n^{1/2}$ as $n\to\infty$. We know \cite{2001Boulton,2006Pravda-Starov} that the distance from these rays to points $z$ on the boundary of the pseudospectra of $H_\vartheta$ increases exactly like $\Re(z)^{1/3}$ for $z\to \infty$. Therefore, despite of the fact that $V$ might be unbounded, the eigenvalues of the non-selfadjoint Schr{\"o}dinger operator $H_\vartheta+V$ eventually lie way in the interior of the pseudospectra of $H_\vartheta$. This is surprising, if we recall that the $\varepsilon$-pseudospectrum is the union of the spectra of all bounded perturbations of $H_\vartheta$  with norm less than (or equal) $\varepsilon$.

The study of eigenvalue asymptotics of non-selfadjoint Schr{\"o}dinger operators has attracted interest from different communities in recent years, see \cite{2018Frank} and references therein. Since condition on the decay of the Schatten-von Neumann norm of the resolvent at infinity are related to conditions of integrability of the semigroup at small times via the Laplace transform, our approach is closer to the framework of relative Schatten-von Neumann class perturbation developed in \cite{2009DemuthHansmannKatriel}. 

Our final statement,  Corollary~\ref{Co:asymptoticperturbedHO}, gives an indication about the shape of the pseudospectra of $H_\vartheta+V$. 
We show that \[
    \lim_{\rho\to \infty} \| (H_\vartheta+V -
    \e^{\pm \ci \vartheta} \rho-\beta)^{-1}\|=0 
    \]
 for all $\beta\in \mathbb{R}$. 
Therefore the distance from the real axis to points $z$ on the boundary of the pseudospectra of $H_\vartheta+V$ is $o(\Re(z))$ as $z\to \infty$. This complements findings in \cite{1999Davies,1999Zworski,2007Pravda-Starov,2004Denckeretal,2017KrejcirikSiegl}, about the resolvent norm growth for non-selfadjoint Schr\"odinger operators with potentials large at infinity. Our results refine in various ways those published in \cite[Chapter~3]{2001Boulton2} many years ago.

I have to thank A.~Doiku and P.~Siegl with whom I sustained a number of useful discussions. Also D.~Krej{\v c}i{\v r}{\'\i}k and J.~Viola for their valuable comments.

\section{Gibbs semigroups and their perturbations} \label{sec1}

Let $\mathcal{H}$ be a Hilbert space.  Below the operator $T$ acting on $\mathcal{H}$ is said to be the generator of a $C_0$-semigroup $\e^{-Tt}$ for $t>0$, if $-T$ is so in the usual sense \cite[Chapter~X and \S10.6]{1957HillePhillips}. We are only concerned about $C_0$-semigroups of compact operators and begin by briefly recalling various well-known facts. 

Let $\e^{-Tt}$ be compact for all $t>0$. Then $\e^{-Tt}$ is continuous in the uniform operator topology for all $t>0$, the resolvent of $T$ is compact and
\begin{equation} \label{eq:spectheo}
     \operatorname{Spec}(\e^{-Tt})=\{\e^{-\lambda t}\,:\,\lambda \in \operatorname{Spec}(T)\}\cup\{0\}.
\end{equation}  
Let
\[
     \varphi(T)=\inf\{a \in \mathbb{R}\,:\,\exists\,M>0,\ \|\e^{-Tt}\|\leq M\e^{a t}\ \forall t>0\}.
\] 
For any $t_0>0$ \cite[pro.1.2.2]{1996vanNeerven},
\[
     \varphi(T)=\frac{\log \operatorname{rad}(\e^{-T t_0})}{t_0}
     =\lim_{t\to\infty} \frac{\log \|\e^{-Tt}\|}{t}.
\]
Here $\operatorname{rad}(W)$ is the spectral radius of $W$. Combining this with  \eqref{eq:spectheo} yields
\begin{equation*} \label{eq:spectralabscisacompact}
    \varphi(T)=-\inf\{\Re(\lambda)\,:\, \lambda\in\operatorname{Spec}(T)\}.
\end{equation*}
That is, the spectral bound and the uniform growth bound of $\e^{-Tt}$ coincide, due to compactness. This property will play an important role below.

The following statement is not a direct consequence of classical results such as the Hille-Yosida theorem. It will serve our purposes later on, hence we include a self-contained proof. Many more precise asymptotic properties of similar nature are known, \emph{cf.}  \cite{1996vanNeerven} and \cite[Ch. 8]{2007Davies}.  

\begin{Th} \label{Th:resolventnormatinfinity}
Let $T$ be the generator of a $C_0$-semigroup such that $\e^{-Tt}$ is compact for all $t>0$.
Then for all $r<-\varphi(T)$ fixed, 
\[
     \lim_{y\to \pm\infty} \|(T-r-\ci y)^{-1}\|=0.
\] 
\end{Th}
\begin{proof}
 Without loss of generality we can assume that $\varphi(T)=0$. Taking inverse Laplace transform \cite[thm. 2.8]{1980Davies}, we get
\begin{equation} \label{eq:Laplacetransform}
    (T-z)^{-1}f=\int_0^\infty \e^{zt}\e^{-Tt}f \,\mathrm{d}t
\end{equation}
for all $f\in\mathcal{H}$ and $z$ such that $\Re(z)<0$. Let $z=r+\ci y$, let $M>1$ be such that $\|\e^{-Tt}\|\leq M$ for all $t>0$ and let 
 \[
     s>-\frac{\log M}{r}>0. 
 \] 
Then
 \begin{align*}
     \e^{-zs}(T-z)^{-1}f &= \int_0^s \e^{z(t-s)}\e^{-Tt}f\,\mathrm{d}t+\int_s^\infty
     \e^{z(t-s)}\e^{-Tt}f\,\mathrm{d}t \\
     &=\int_0^s \e^{z(t-s)}\e^{-Tt}f\,\mathrm{d}t+\int_0^\infty
     \e^{zt}\e^{-T(t+s)}f\,\mathrm{d}t \\
     &=\int_0^s \e^{z(t-s)}\e^{-Tt}f\,\mathrm{d}t+\e^{-Ts}(T-z)^{-1}f
 \end{align*}
 for all $f\in\mathcal{H}$. Applying the triangle inequality and then placing all the resolvent norms to the left hand side, yields
 \[
     \|(T-r-\ci y)^{-1}\|\leq \frac{\e^{-rs}}{\e^{-rs}-M}\left\|
     \int_0^s \e^{\ci yt}\e^{rt}\e^{-Tt}\,\mathrm{d}t\right\|.
 \]
Here the constant outside the norm is positive and does not depend upon $y$. 

Since $\e^{-Tt}$ is continuous in the uniform operator topology, it is also locally integrable (Riemann and Bochner) with respect to the associated norm. Let
\[
      \hat{F}(y)=\int_{0}^s \e^{-\ci y t} \e^{rt}\e^{-Tt}\,\mathrm{d}t.
\]
Then $\hat{F}(y)$ is the Fourier transform of the operator-valued function \[t\mapsto \e^{rt}\e^{-Tt}\xi_{[0,s]}(t)\] which lies in $L^1(\mathbb{R};\mathcal{B}(\mathcal{H}))$. Here $s$ is fix. Thus, a version of the Riemann-Lebesgue lemma for Bochner spaces (the proof is identical to the classical result \cite[sec. 7.2]{1992Folland} as the integrand above is a limit of step functions) ensures that $\|\hat{F}(y)\|\to 0$ as $y\to\pm\infty$.
\end{proof}

Let $q\geq 1$. We denote by $\mathcal{C}_q$ the $q$ Schatten-von Neumann operator ideal and by $\|\cdot\|_q$ the corresponding norm. As usual, here $\mathcal{C}_{\infty}$ is the compact operators. Below we write $\|\cdot\|_{\infty}\equiv \|\cdot\|$ for operators.

By virtue of the semigroup property, $\e^{-T(s+t)}=\e^{-Ts}\e^{-Tt}$,  combined with the fact that $\|\cdot\|_{q}$ is non-increasing as $q$ increases, it follows that $\{\e^{-Tt}\}_{t>0}\subset \mathcal{C}_q$ for some $q<\infty$ if and only if $\{\e^{-Tt}\}_{t>0}\subset \mathcal{C}_1$. A $C_0$-semigroup with this property is often called a Gibbs semigroup, \cite{1971Uhlenbrock, 1974Angelescuetal, 1989Zagrebnov}. We will adhere to this terminology.
 
If $\e^{-Tt}$ is a Gibbs semigroup, then $t\mapsto \e^{-Tt}$ is continuous in the trace norm $\|\cdot\|_1$ for all $t>0$.  If the generator is unbounded, the $C_0$-semigroup is always discontinuous at $t=0$ in $\|\cdot\|_{\infty}$ and hence in all of the other norms $\|\cdot\|_q$. We now determine a class of perturbations of the generator which preserve the finite trace.

A closed operator $A$ is some times  said to be a class~$\mathcal{P}$ perturbation of the generator $T$ iff
\[
     \operatorname{D}(A)\supset \bigcup_{t>0} \e^{-Tt}(\mathcal{H})
\qquad \text{ and } \qquad
     \int_0^1 \|A \e^{-Tt}\|_{\infty} \mathrm{d}t<\infty,
\]
see \cite[p70]{1980Davies}. If $A$ is a class~$\mathcal{P}$ perturbation of $T$, then $\operatorname{D}(A)\supset \operatorname{D}(T)$, $A$ is $T$-bounded with relative bound equal to zero and the closed operator $(T+A)$ is the generator of a $C_0$-semigroup on $\mathcal{H}$.  The perturbed $C_0$-semigroup is given in terms of the unperturbed one via a  Dyson-Phillips expansion
\begin{equation} \label{eq:Dysonexpansion}
       \e^{-(T+A)t}=\sum_{k=0}^\infty (-1)^k W_k(t) 
\end{equation}
which is absolutely convergent in $\|\cdot\|_\infty$ for all $0<t\leq a$ where
\[
    \int_0^a \|A\e^{-Tt}\|_\infty \mathrm{d} t<1
    \]
and    
\begin{equation} \label{eq:Dysonterms}
\begin{aligned}
W_0(t)&=\e^{-Tt} \\
W_k(t)&=\int_{s=0}^t W_{k-1}(t-s) A\e^{-Ts} \ \mathrm{d}s. 
\end{aligned}
\end{equation}
The integrals are convergent in $\|\cdot\|_\infty$ for all $t>0$.
From this it follows that the variation of parameter formula
\begin{equation}  \label{eq:variationparameters}
    \e^{-(T+A)t}=\e^{-Tt}+\int_{0}^t \e^{-(T+A)(t-s)}A\e^{-Ts} \mathrm{d}s
\end{equation}
holds true, where the integral converges in $\|\cdot\|_\infty$ for all $0<t\leq a$.

If $T$ is m-sectorial and $\e^{-\Re(T)t}$ is a Gibbs semigroup, then  $\e^{-(T+A)t}$ is also a Gibbs semigroup whenever $A$ is a class~$\mathcal{P}$ perturbation of $T$, see \cite{1974Angelescuetal,1989Zagrebnov,2001CachiaZagrebnov}. The following example shows that these hypotheses cannot be weakened that easily. Let $b\in \mathbb{R}$ and $\{e_n\}_{n=1}^\infty$ be an orthonormal basis of $\mathcal{H}$. Let
\[
    T=\sum_{n=1}^\infty (\ci n^3+n)|e_n\rangle \langle e_n|
\qquad \text{and} \qquad A\equiv A_b=\sum_{n=1}^\infty bn|e_n\rangle \langle e_n|
    \]
in their maximal domains. Then, $T+A_b$ is: the generator of a Gibbs semigroup for $b>-1$, only the generator of a unitary group for $b=-1$ and
not even a generator of a $C_0$-semigroup for $b<-1$. This, despite of the fact that $A_bT^{-1}$ is trace class for all $b\in \mathbb{R}$. Note that $\|A_b\e^{-Tt}\|_\infty\sim t^{-1}$ as $t\to 0$, so $A_b$ is not a class $\mathcal{P}$ perturbation of $T$. Also that $A_b$ destroys the m-sectoriality of $T$ for $b<-1$.

In the next lemma, $T$ is not assumed to be m-sectorial and $\Re(T)$ might not be the generator of a compact semigroup. The proof follows closely the line of arguments in \cite[theorems~3.1-3.5]{1980Davies}, replacing the operator norm with the norm of $\mathcal{C}_q$. In this proof we could have used directly the variation of parameters formula (as we do later on), but we prefer to highlight the range of absolute convergence of the Dyson-Phillips expansion in $\|\cdot\|_q$. 
 
\begin{Le} \label{Th:classPCpperturbation}
Let $T$ be the generator of a $C_0$-semigroup such that $\e^{-Tt}\in \mathcal{C}_1$ for all $t>0$. Let $A$ be a closed operator such that
\begin{equation} \label{eq:classPCqpert}
     \operatorname{D}(A)\supset \bigcup_{t>0} \e^{-Tt}(\mathcal{H})
\qquad \text{ and } \qquad
     \int_0^{1} \|A \e^{-Tt}\|_q \mathrm{d}t<\infty 
     \end{equation}
for some $q<\infty$. Then $T+A$ with domain $\mathrm{D}(T)$ is also the generator of a  $C_0$-semigroup such that $\e^{-(T+A)t}\in \mathcal{C}_1$ for all $t>0$. Moreover
\begin{equation} \label{eq:asymptato}
      \|\e^{-(T+A)t}\|_q=O(\|\e^{-Tt}\|_q) \qquad t\to 0^+.
\end{equation}     
\end{Le} 
\begin{proof}
 The hypotheses ensure that $A$ is a class~$\mathcal{P}$ perturbation of $T$. Hence $T+A$ with domain $\mathrm{D}(T)$ is the generator of a $C_0$-semigroup. Moreover $\e^{-(T+A)t}$ is given by \eqref{eq:Dysonexpansion} convergent in $\|\cdot\|_{\infty}$ for $t>0$ small enough. By considering $T-\varphi(T)$ instead of $T$ for the general case, we can assume without loss of generality that $\|\e^{-Tt}\|_{\infty}\leq M$ for all $t> 0$.

Let us show first that in \eqref{eq:Dysonterms}, $W_k(t)\in\mathcal{C}_q$ for all $k\in \mathbb{N}$ and $t>0$. We begin with $k=1$. Since 
\begin{align*}
    \int_{n-1}^n\|A\e^{-Ts}\|_q\, \mathrm{d}s&=
    \int_{0}^1\|A\e^{-T(s+n-1)}\|_q\, \mathrm{d}s \\
    &\leq M\int_0^1 \|A\e^{-Ts}\|_q\mathrm{d}s<\infty
    \qquad \forall n\in \mathbb{N}.
\end{align*}
Then
\[
    \int_{s=0}^t\|A\e^{-Ts}\|_q\,\mathrm{d}s<\infty \qquad \forall t>0. 
\]
Fix $t>0$.  Then 
\[
   \int_{s=0}^t \|\e^{-T(t-s)}A\e^{-Ts}\|_q\,\mathrm{d}s \leq
   M\int_{s=0}^t \|A\e^{-Ts}\|_q\,\mathrm{d}s<\infty.
   \]
   The $\mathcal{C}_q$-valued function $s\mapsto \e^{-Ts}$ is continuous in $\|\cdot\|_q$ for all $s>0$, because it is continuous in the trace norm.  Then the $\mathcal{C}_q$-valued function $s\mapsto \e^{-T(t-s)}A\e^{-Ts}$ is also continuous with respect to the norm $\|\cdot\|_q$. Indeed, for fix $b>0$ such that $b<\frac{t}{2}$ and for $s>\frac{t}{2}$, we have
\[
   \|A\e^{-Tt}-A\e^{-Ts}\|_q\leq \|A\e^{-Tb}\|_\infty \|e^{-T(t-b)}-\e^{-T(s-b)}\|_q\to 0 \qquad s\to t.
   \]
Therefore the integrand in the expression for $W_1(t)$ is Riemann integrable in $\mathcal{C}_q$ in all segments of the form $[\alpha,1]$ for $\alpha>0$. Note that this integral is improper in the norm $\|\cdot\|_q$ in the segment $(0,1]$ but the right hand side of \eqref{eq:classPCqpert} ensures that this improper integral is convergent.  Hence  $W_1(t)\in \mathcal{C}_q$
 and
 \[
      \|W_1(t)\|_q\leq M\int_{s=0}^t \|A\e^{-Ts}\|_q\,\mathrm{d}s.
      \]

 Now consider $k=2$. Let
 \[
      F(x)=\begin{cases} \e^{-Tx} &x>0 \\
           0 & \text{otherwise}. \end{cases}
 \]
 Then
 \begin{align*}
 \int_{s=0}^t\int_{u=0}^s &\|\e^{-T(t-s)}A\e^{-T(s-u)}A\e^{-Tu} \|_q\  \mathrm{d}u\ \mathrm{d}s \\ &\leq M \int_{s=0}^t\int_{u=0}^s \|AF(s-u)AF(u) \|_q\  \mathrm{d}u\ \mathrm{d}s\\
 &= M \int_{s=0}^t\int_{u=0}^t \|AF(s-u)AF(u) \|_q\  \mathrm{d}u\ \mathrm{d}s \\
 &\leq M \int_{u=0}^t\int_{s=0}^t \|AF(s-u)\|_{\infty}\|AF(u) \|_q\  \mathrm{d}s\ \mathrm{d}u \\
 &\leq M \int_{u=0}^t \|AF(u)\|_q\int_{s=0}^t\|AF(s-u)\|_{q}  \mathrm{d}s\ \mathrm{d}u \\
 &= M \int_{u=0}^t \|AF(u)\|_q\int_{x=-u}^{t-u}\|AF(x)\|_{q}  \mathrm{d}x\ \mathrm{d}u \\
 &\leq M \int_{u=0}^t \|AF(u)\|_q\ \mathrm{d}u \int_{x=0}^{t}\|AF(x)\|_q  \mathrm{d}x \\
 &= M\left[ \int_{s=0}^t \|A\e^{-Ts}\|_q\ \mathrm{d}s   \right]^2<\infty.
 \end{align*}
 Hence, by continuity, the $\mathcal{C}_q$-valued function $(s,u)\mapsto \e^{-T(t-s)}A\e^{-T(s-u)}A\e^{-Tu}$ is integrable (Riemann with the improper integral once again convergent) with respect to the norm $\|\cdot\|_q$  in the region $0<u<s<t$. Thus the integral $W_2(t)$ in \eqref{eq:Dysonterms} also converges in the norm of $\mathcal{C}_q$, $W_2(t)\in \mathcal{C}_q$ and
 \[
    \|W_2(t)\|_q\leq M\left[ \int_{s=0}^t \|A\e^{-Ts}\|_q\ \mathrm{d}s   \right]^2.
    \]
    
Similar arguments show that all $W_k(t)\in\mathcal{C}_q$ and
 \[
          \|W_k(t)\|_q\leq M\left[ \int_{s=0}^t \|A\e^{-Ts}\|_q\ \mathrm{d}s   \right]^k\qquad \forall t>0,\, k=3,4,\ldots
 \]
 
 In order to show that $\e^{-(T+A)t}\in\mathcal{C}_q$, it is then enough to prove the convergence in the norm of $\mathcal{C}_q$ of the series at the right hand side of \eqref{eq:Dysonexpansion} for $t>0$ small enough. Choose $a>0$ such that 
 \[
      \int_{s=0}^{a} \|A\e^{-Ts}\|_q \mathrm{d}s<1.
 \]
 Then for all $t\in(0,a]$ the series $\sum_{k=1}^{\infty}\|W_k(t)\|_q<\infty$. This guarantees the convergence of the right hand side of \eqref{eq:Dysonexpansion} and $\e^{-(T+A)t}\in\mathcal{C}_q$ for $0<t\leq a$. The latter conclusion for $t>a$ is a consequence of the semigroup property. Hence $\e^{-(T+A)t}$ is also a Gibbs semigroup.

Finally, note that for $0<t\leq a$ in the above calculation, we have
\begin{align*}
      \|\e^{-(T+A)t}\|_q&\leq \|\e^{-Tt}\|_q+\sum_{k=1}^\infty \|W_k(t)\|_q \\
      &\leq \|\e^{-Tt}\|_q+M\sum_{k=1}^\infty \left(\int_0^a \|A\e^{-Ts}\|_q \mathrm{d} s\right)^k.
\end{align*}
As the series at the right hand side converges independent of $t$, then there exists $\tilde{M}>0$ independent of $t$, such that
\[
      \|\e^{-(T+A)t}\|_q\leq \|\e^{-Tt}\|_q+\tilde{M} \qquad \qquad 0<t\leq a.
\]   
\end{proof}

If $T$ and $A$ satisfy the hypothesis of Lemma~\ref{Th:classPCpperturbation}, then the improper integral in the variation of parameters formula \eqref{eq:variationparameters} converges in $\|\cdot\|_q$. Indeed the map $s\mapsto \e^{-(T+A)(t-s)}A\e^{-Ts}$  is $\|\cdot\|_q$ continuous and
\[
\left\| \e^{-(T+A)(t-s)}A\e^{-Ts}\right\|_q\leq
\left\| \e^{-(T+A)(t-s)}\right\|_\infty \left\|A\e^{-Ts}\right\|_q
\]
where
\[
     \left\| \e^{-(T+A)(t-s)}\right\|_\infty=O\left(\left\| \e^{-T(t-s)}\right\|_\infty\right)=O(1) \qquad s\to 0 \text{ and } s\to t.
\]

\begin{Rm} \label{Rm:connectionotherresults}
Both the results of \cite{1974Angelescuetal} and those of \cite{1989Zagrebnov} concerning perturbations of m-sectorial generators, are consequence of an inequality originally found by Ginibre and Gruber \cite{1969GinibreGruber} extended from the selfadjoint setting. Details apparently missing in \cite{1989Zagrebnov} were completed in \cite{2001CachiaZagrebnov}. In the latter, this extension was formulated for m-sectorial operators.  Unfortunately we do not have an analogue inequality at hand under the more general hypothesis above.
\end{Rm}  

Lemma~\ref{Th:classPCpperturbation} induces the following terminology which will simplify the discussions below.

 \begin{De} \label{De:classPCq}
Let $1\leq q \leq \infty$. The closed operator $A$ is said to be a class~$\mathcal{PC}_q$ perturbation of the generator $T$ of a $C_0$-semigroup $\{\e^{-Tt}\}_{t>0}\subset \mathcal{C}_1$, if \eqref{eq:classPCqpert} are satisfied.
\end{De}

 If $A$ is a class~$\mathcal{PC}_q$ perturbation of $T$, it is also a  class~$\mathcal{PC}_p$ perturbation of $T$ for all $p>q$.

If two closed operators $A_1$ and $A_2$ are class~$\mathcal{PC}_q$ perturbations of the generator of a Gibbs semigroup, it is not necessarily the case that the sum $A_1+A_2$ (on a suitable domain) is closable. For this reason, the class described in Definition~\ref{De:classPCq} is not additive. By following the ideas of \cite[\S13.3-13.5]{1957HillePhillips}, it is possible to extend this definition to perturbations that are not necessarily closable, then obtain an additive class and an equivalence relation for generators. The details of this require developing extra notation that will not serve our focused purpose in the next section when considering $T=H_\theta$. Therefore we do not address this for the time being.

Now an example. Let $T=T^*$ be the selfadjoint operator with compact resolvent given by
\[
    T=\sum_{n=1}^\infty n|e_n\rangle \langle e_n|,
    \]
    where $\{e_n\}_{n=1}^\infty$ is an orthonormal basis of $\mathcal{H}$.
    Then
\[
    \e^{-Tt}=\sum_{n=1}^\infty \e^{-nt}|e_n\rangle \langle e_n|.
    \]
    Hence
\[
    \|\e^{-Tt}\|_1=\frac{\e^{-t}}{1-\e^{-t}}    
    \]
    and $\e^{-Tt}$ is a Gibbs semigroup. For $\alpha\leq 1$, let
    \[
       A_\alpha=T^{\alpha}=\sum_{n=1}^\infty n^{\alpha}|e_n\rangle \langle e_n|
       \]
    in its maximal domain.   
    Then
    \[
    \|A_\alpha\e^{-Tt}\|_q=\begin{cases} \max_{n\in \mathbb{N}}n^{\alpha}\e^{-tn} & q=\infty \\
    \left( \sum_{n=1}^\infty n^{\alpha q}\e^{-t q n}    \right)^{1/q} & 1\leq q <\infty \end{cases}
    \]
    For $q=\infty$, we have $\|A_\alpha \e^{-Tt}\|_\infty \sim t^{-\alpha}$ as $t\to 0^+$. Then $A_\alpha$ is a class~$\mathcal{PC}_\infty$ (class~$\mathcal{P}$) perturbation of $T$ for all $\alpha<1$. For $q<\infty$,
\[
          \|A_\alpha\e^{-Tt}\|^q_q=\mathrm{Li}_{(-\alpha q)}(\e^{-tq})
          \]
          where $\mathrm{Li}_s(z)$ is the polylogarithm function. Since
          \[
\lim_{z\to 1}(1-z)^{1-s}\operatorname{Li}_s(z)=\Gamma(1-s) \qquad \forall s<1 
          \]
          \cite[9.557]{2007GradshteynRyzhik}, for all $q\alpha>-1$
          \[
\|A_\alpha\e^{-Tt}\|_q\sim t^{-\frac{q\alpha +1}{q}} \qquad t\to 0^+. 
\]
Then, $A_\alpha$ is a class~$\mathcal{PC}_q$ perturbation of $T$ if and only if
$q>\frac{1}{1-\alpha}$ (assuming $q\geq 1$ as in the definition above).  This shows that, the smaller the $q$, the ``multiplicative smaller'' the perturbation of a generator of a Gibbs semigroup should be, in order to be included in the class~$\mathcal{PC}_q$. It also shows that, although they are nested, these classes are not equal in general. Note that for $\alpha=0$, $A_0=I$ is not a class~$\mathcal{PC}_q$ perturbation of $T$ for $q=1$, but it is so for all $q>1$. We can relate this example to the harmonic oscillator by taking $T=\frac12(H_0+1)$.

Let us now determine that the Definition~\ref{De:classPCq} is symmetric. The next lemma follows the template of \cite[Lemma~13.5.1]{1957HillePhillips}.

\begin{Le} \label{Le:consistencyequivalence}
  Let $T$ be the generator of $\e^{-Tt}\in \mathcal{C}_1$ for all $t>0$. If $A_1$ and $A_2$ are two closed operators such that they are both class~$\mathcal{PC}_q$ perturbations of $T$, then
  \[
     \operatorname{D}(A_2)\supset \bigcup_{t>0} \e^{-(T+A_1)t}(\mathcal{H})
\qquad \text{ and } \qquad
     \int_0^{1} \|A_2 \e^{-(T+A_1)t}\|_q \mathrm{d}t<\infty    
  \]
\end{Le}
\begin{proof}
Since $A_j$ are class~$\mathcal{P}$ perturbations of $T$, by virtue of \cite[Lemma~13.5.1]{1957HillePhillips}\footnote{In the notation of \cite{1957HillePhillips} this is written as  $A_j\!\!\upharpoonright_{\mathrm{D}(-T)}\in \mathfrak{B}(-T)$ and here we are also invoking \emph{loc. cit.} Theorem~13.3.1.}, we know that
     \[
     \operatorname{D}(A_2)\supset \bigcup_{t>0} \e^{-(T+A_1)t}(\mathcal{H})
     \]
     as required in the first part of the conclusion. Moreover
     \begin{equation} \label{eq:inequalityforintegrability}
         \int_0^{1} \|A_2 \e^{-(T+A_1)t}\|_{\infty} \mathrm{d}t<\infty. 
         \end{equation}

In order to show the second part of the conclusion, we use the variation of parameters formula. From Lemma~\ref{Th:classPCpperturbation}, it follows that $\e^{-(T+A_j)t}\in \mathcal{C}_1$ for all $t>0$. Also, we know that
  \[
     \e^{-(T+A_1)t}=\e^{-Tt}+\int_{0}^t \e^{-(T+A_1)(t-s)}A_1\e^{-Ts} \mathrm{d}s
     \]
     where the integral converges in $\|\cdot\|_q$ (for $t$ small enough). Since all the improper integrals involved in the following expression are Riemann integrals and they are convergent in $\|\cdot\|_{\infty}$ and since the operator $A_2$ is closed, we have
\begin{equation} \label{eq:variationparameterswithA}
     A_2\e^{-(T+A_1)t}=A_2\e^{-Tt}+\int_{0}^t A_2\e^{-(T+A_1)(t-s)}A_1\e^{-Ts} \mathrm{d}s,
     \end{equation}
     see \cite[Theorem~3.3.2]{1957HillePhillips}. Also,
     $(s,t)\mapsto  A_2\e^{-(T+A_1)(t-s)}A_1\e^{-Ts}$ is continuous in $\|\cdot\|_q$. Moreover,
     \begin{align*}
       \int_{t=0}^1& \left\| \int_{s=0}^t A_2 \e^{-(T+A_1)(t-s)}A_1\e^{-Ts} \mathrm{d}s\right\|_q\mathrm{d}t \\ &\leq \int_{t=0}^1 \int_{s=0}^t \left\|A_2 \e^{-(T+A_1)(t-s)}A_1\e^{-Ts} \right\|_q \mathrm{d}s\mathrm{d}t \\
       & \leq  \int_{t=0}^1 \int_{s=0}^t \left\|A_2 \e^{-(T+A_1)(t-s)}\right\|_\infty \left\|A_1\e^{-Ts} \right\|_q \mathrm{d}s\mathrm{d}t \\
       &=  \int_{t=0}^1 \int_{s=0}^1 \left\|A_2 F_1(t-s)\right\|_\infty \left\|A_1F(s) \right\|_q \mathrm{d}s\mathrm{d}t \\
         &=  \int_{s=0}^1 \left\|A_1 F(s)\right\|_q \int_{t=0}^1  \left\|A_2F_1(t-s) \right\|_{\infty} \mathrm{d}t\mathrm{d}s \\
         &\leq   \int_{s=0}^1 \left\|A_1 F(s)\right\|_q \mathrm{d}s \int_{x=0}^1  \left\|A_2F_1(x) \right\|_{\infty} \mathrm{d}x.
     \end{align*}
     Here we write $F(x)$ as in the proof of Lemma~\ref{Th:classPCpperturbation} and
\[
      F_1(x)=\begin{cases} \e^{-(T+A_1)x} &x>0 \\
           0 & \text{otherwise}. \end{cases}
      \]
The hypothesis and \eqref{eq:inequalityforintegrability}, yield that the this double integral is finite. Hence the second conclusion follows from this, integrating \eqref{eq:variationparameterswithA}.      
\end{proof}

\begin{Co} \label{Co:equalasymptotics}
  Let $q\geq 1$. Let $A$ be a class~$\mathcal{PC}_q$ perturbation of the generator $T$ of a Gibbs semigroup. Then
  \[
     \|\e^{-(T+A)t}\|_q\sim \|\e^{-Tt}\|_q \qquad t\to 0^+.
  \]
\end{Co}
\begin{proof}
Let $T_2=T+A$ with $\mathrm{D}(T_2)=\mathrm{D}(T)$. Then $-A$ is a class~$\mathcal{PC}_q$ perturbation of $T_2$ as a consequence of Lemma~\ref{Le:consistencyequivalence} with $A_1=-A=A_2$.
\end{proof}

If $A$ is both accretive and $T$-bounded with bound less than one,  then $T+A$ is the generator of a $C_0$-semigroup \cite[Corollary~3.8]{1980Davies}. In Lemma~\ref{Th:classPCpperturbation}, the perturbation $A$ is allowed to be non-accretive, at the cost of being relatively compact (and more). See Remark~\ref{remark:dissipativityofthepert} below.

\begin{Le} \label{Th:perofclassPCqpert}
Let $A$ be a class~$\mathcal{PC}_q$ perturbation of the generator $T$.  If $0\not \in \operatorname{Spec}A$, then the closure of any other closable operator $B$ such that $\mathrm{D}(B)\supset \mathrm{D}(A)$ is also a class~$\mathcal{PC}_q$ perturbation of  $T$.
\end{Le}  
\begin{proof}
Let $\overline{B}$ be the closure of $B$. The inclusion of the domains and the close graph theorem ensure that $\overline{B}$ is $A$-bounded \cite[p.191]{1980Kato}. Then 
\[
    \|\overline{B}\e^{-Tt}\|_q\leq \|\overline{B}A^{-1}\|_{\infty}\|A\e^{-Tt}\|_q.
    \]
Hence $\overline{B}$ also satisfies the right hand side of \eqref{eq:classPCqpert}.
\end{proof}

Our major objective after this section will be to apply the framework just introduced to the holomorphic semigroup generated by the non-selfadjoint harmonic oscillator and perturbations by potentials. If $T$ is the generator of a bounded holomorphic semigroup on a sector and $A$ is $T$-bounded with relative bound equal to $0$, then  $T+A+c$ is the generator of a bounded holomorphic semigroup on that sector for some $c>0$, \cite[Corollary~2.5, p.500]{1980Kato}. If $A$ is additionally a class~$\mathcal{PC}_q$ perturbation of the generator $T$ of a Gibbs semigroup, as we shall see next, the small $t$ asymptotic behaviour of the $\mathcal{C}_q$ norm is preserved even at the boundary of the sector.

For $\alpha,\beta\in(0,\frac{\pi}{2}]$, here and elsewhere we write
\[
     \mathcal{S}(-\alpha,\beta)=\{r\e^{\ci\omega}\,:\, r>0,\,
     \omega\in(-\alpha,\beta)\}.
     \]
Let $T$ be an m-sectorial operator. Then, $\e^{-T\tau}$ is a bounded holomorphic semigroup  for all $\tau\in \mathcal{S}(-\alpha,\beta)$ with suitable $\alpha$ and $\beta$. If $\e^{-Tt}\in \mathcal{C}_1$ for all $t>0$, then also $\e^{-T\tau}\in \mathcal{C}_1$ for all $\tau\in \mathcal{S}(-\alpha,\beta)$ and $\tau\mapsto \e^{-T\tau}$ is  holomorphic in $\mathcal{S}(-\alpha,\beta)$ with respect to $\|\cdot\|_1$. For $\theta =-\alpha$ or $\theta=\beta$, the $C_0$-semigroup $\e^{-\e^{\ci \theta}Tt}$ might or might not be compact. It is not compact for example, whenever $T=T^*>0$ and $\alpha=\beta=\frac{\pi}{2}$. But, as we shall see in the next section, some times $\e^{-T\tau}\in \mathcal{C}_1$ for all $\tau\in \overline{\mathcal{S}(-\alpha,\beta)}\setminus\{0\}$, the maximal sector of analyticity. By applying Corollary~\ref{Co:equalasymptotics} to rotations of the operators involved, it is straightforward that class~$\mathcal{PC}_q$ perturbations preserve this characteristic.

\begin{Th} \label{Th:Gibbsinsector}
Let $T$ be the generator of a semigroup $\e^{-T\tau} \in\mathcal{C}_1$ for all $\tau\in \overline{\mathcal{S}(-\alpha,\beta)}\setminus \{0\}$ holomorphic in $\mathcal{S}(-\alpha,\beta)$. If $A$ is a class 
$\mathcal{PC}_q$ perturbation of $T$ for $q<\infty$, then $T+A$ is also the generator of a semigroup $\e^{-(T+A)\tau} \in\mathcal{C}_1$ for all $\tau\in \overline{\mathcal{S}(-\alpha,\beta)}\setminus\{0\}$ holomorphic in  $\mathcal{S}(-\alpha,\beta)$. Moreover, for all $-\alpha \leq \theta \leq \beta$,
\[
\|\e^{-(T+A)\e^{\ci \theta} r}\|_q\sim\|\e^{-T\e^{\ci \theta} r}\|_q \qquad r \to 0.
        \]        
\end{Th}

See also Theorem~\ref{Th:theopert} below.

\section{Asymptotic behaviour of the non-selfadjoint Mehler kernel}
\label{sec2}
The numerical range of $H_\vartheta$ is
\[
    \operatorname{Num}(H_{\vartheta})=\left\{\e^{-\ci\vartheta}s+\e^{\ci\vartheta}t\,:\, s,\,t\in \mathbb{R},\,st\geq \frac14     \right\}\subset \mathcal{S}(-|\vartheta|,|\vartheta|),
\]
\cite[pro. 2.1]{2001Boulton}. Then $H_{\vartheta}$ is m-sectorial and the generator of a bounded holomorphic semigroup $\e^{-H_{\vartheta}\tau}$ for all 
\[
\tau\in \mathcal{S}_{\vartheta}\equiv \mathcal{S}\left(-\frac{\pi}{2}+|\vartheta|,\frac{\pi}{2}-|\vartheta| \right).
\]
Moreover $\e^{\ci (\pm\frac{\pi}{2}\mp|\vartheta|)}H_{\vartheta}$ are generators of $C_0$-semigroups for all $\vartheta\in(-\frac{\pi}{2},\frac{\pi}{2})$. Whenever $\vartheta\not=0$, $\e^{-H_{\vartheta}\tau}$ is continuous in $\|\cdot\|_{\infty}$ for all 
$
\tau\in \overline{\mathcal{S}_{\vartheta}}\setminus\{0\}.
$
 This is not the case for $\vartheta=0$ and $\tau$ approaching the boundary of the segment $\mathcal{S}(-\frac{\pi}{2},\frac{\pi}{2})$, because $\e^{\pm\ci H_0 t}$ are unitary groups for $t\in\mathbb{R}$. 

According to the framework of \cite{2017AlemanViola,2016Viola}, when seen as a family of bounded operators in $\tau$, the holomorphic semigroup $\e^{-H_\theta \tau}$ has a bounded extension (in the uniform operator norm) to the maximal semi-modulus
\[
     \mathcal{T}_{\vartheta}=\left\{\tau\in \mathbb{C}:\Re \tau>0,\,|\arg \tanh(\tau)|<\frac{\pi}{2}-|\vartheta|\right\}\supset \mathcal{S}_{\vartheta}.
\]
This extension is analytic and compact for all $\tau\in \mathcal{T}_{\vartheta}$, and it is bounded for all $\tau\in \overline{\mathcal{T}_{\vartheta}}$. The operator $H_\vartheta$ which has Weyl symbol $q_{\vartheta}(x,\zeta)=\e^{-\vartheta}\zeta^2+\e^{\vartheta}x^2$, corresponds to that presented in \cite[Example~2.1]{2016Viola}. 

We now determine various asymptotic properties of $\e^{-H_{\vartheta}\tau}$ in parts of this maximal region. Let
\begin{align*} 
  \lambda &\equiv \lambda(\tau)=\e^{-2 \tau} \\
  w_1 & \equiv w_1(\vartheta,\tau)=\e^{\ci \vartheta}
  \left[\frac{ \lambda(\tau)}{1-\lambda^2(\tau)}\right]=\frac{\e^{\ci \vartheta}}{2}\operatorname{csch}(2\tau)  \\
  w_2& \equiv w_2(\vartheta, \tau)= 
  \frac{\e^{\ci\vartheta}}{2}\left[\frac{1+\lambda^2(\tau)}{1-\lambda^2(\tau)}\right] =\frac{\e^{\ci\vartheta}}{2}\coth(2\tau).
\end{align*}
and
\[
      M_{\vartheta}(\tau,x,y)=\left(\frac{w_1}{\pi}\right)^{1/2}\exp\left[2w_1xy-w_2(x^2+y^2)\right].
\]
The classical Mehler's formula extends to non-real $\tau$ \cite[Theorem~4.2]{2001Boulton}, 
\begin{equation*} \label{eq:Mehlerinsector}
     \e^{-H_{\vartheta}\tau}f(x)=\int_{-\infty}^\infty
     M_{\vartheta}(\tau,x,y)f(y) \mathrm{d}y \qquad \forall \tau\in \mathcal{S}_{\vartheta}.
\end{equation*}
Let $r_j\equiv r_j(\vartheta,\tau)=\Re \left[w_j(\vartheta,\tau)\right]$. 
In the next statement, note that
\begin{equation}   \label{eq:conditionsangles}
\begin{gathered}
   |\omega|\leq\frac{\pi}{2}-|\vartheta| \qquad \Rightarrow \qquad |\cos\theta| \geq |\sin \omega| \\
 \qquad \text{and} \qquad   |\omega|=\frac{\pi}{2}-|\vartheta|  \qquad \iff \qquad  |\cos\theta| = |\sin \omega|.
\end{gathered}
\end{equation}

\begin{Le}   \label{signandasymptoticsofrealpartswj}
The conditions
\begin{equation} \label{eq:forSchattennorm}
     r_2(\vartheta,\tau)>0\quad \text{and} \quad r_2(\vartheta,\tau)\pm r_1(\vartheta,\tau)>0
\end{equation}
hold, if and only if $\tau\in \mathcal{T}_{\vartheta}$.
Moreover, as $t\to 0^+$,
\begin{align*}
     |w_1(\vartheta,\e^{\ci\omega}t)|&= \frac14 t^{-1}+O(1), \\
     r_2(\vartheta,\e^{\ci\omega}t)&=\begin{cases} 
\frac{\cos(\omega+\theta)}{2}t^{-1}+O(1) &  |\omega|<\frac{\pi}{2}-|\vartheta|   \\
 \frac{\sin(4\theta)}{3}t+O(t^2)    & |\omega|=\frac{\pi}{2}-|\vartheta|
     \end{cases}
\end{align*}
and
\[
    \frac{1}{r_2(\vartheta,\e^{\ci\omega}t)^2-r_1(\vartheta,\e^{\ci\omega}t)^2}=\begin{cases} 
 \frac{1}{\cos^2\vartheta-\sin^2 \omega}+O(t^2) &  |\omega|<\frac{\pi}{2}-|\vartheta|   \\
 \frac{3}{\sin^2(2\vartheta)}t^{-2}+O(1)    & |\omega|=\frac{\pi}{2}-|\vartheta|,
     \end{cases}
\]
for fixed $\vartheta \in\left(-\frac{\pi}{2} ,\frac{\pi}{2}\right)$ and $\omega\not =\pm\frac{\pi}{2}$.
\end{Le}
\begin{proof}
For the first part of the lemma we show that
\begin{equation} \label{eq:optimalregionandcoeffMehler1}
  |\arg(w_2\pm w_1)|<\frac{\pi}{2}\quad \iff \quad 
|\arg \tanh \tau|<\frac{\pi}{2}-|\vartheta|
\end{equation}
and that
\begin{equation} \label{eq:optimalregionandcoeffMehler2}
 |\arg \tanh \tau|<\frac{\pi}{2}-|\theta|\quad   \Longrightarrow \quad  
|\arg w_2 |<\frac{\pi}{2}.
\end{equation}
Since
\[
     \tanh \tau=\frac{1-\lambda}{1+\lambda} \quad \text{and} \quad
     w_2\pm w_1=\e^{\ci \vartheta}\frac{1\pm\lambda}{1\mp\lambda},
\]
then
\[
    \arg (w_2\pm w_1)=\vartheta\pm \arg(1+\lambda)\mp \arg(1-\lambda)=
    \vartheta\mp \arg \tanh \tau
\]
and hence \eqref{eq:optimalregionandcoeffMehler1}. Suppose that the left hand side of \eqref{eq:optimalregionandcoeffMehler2} holds true. That is $\tanh \tau\in \mathcal{S}_{\vartheta}$. Then also $\coth \tau \in \mathcal{S}_{\vartheta}$. 
By convexity of the sector, also
\[
      \tanh(2\tau)=\frac{2}{\coth \tau+\tanh \tau}\in
      \mathcal{S}_{\vartheta}. 
\]
Thus, if $\tau\in \mathcal{T}_{\theta}$, also $2\tau \in \mathcal{T}_{\theta}$. Since 
\[
    w_2(\theta,\tau)=w_2(\theta,2\tau)+w_1(\theta,2\tau),
\]
by the equivalence in \eqref{eq:optimalregionandcoeffMehler1} we get that also \eqref{eq:optimalregionandcoeffMehler2} holds true. This completes the first part of the lemma.

In the second part, the proof of the first asymptotic formula is straightforward. For the second and third formulas, let $a=2\cos \omega$ and $b=2\sin \omega$. 
Then
\[
    r_2=\frac{\cos \vartheta \sinh 2at + \sin \vartheta \sin 2bt}{\cosh 2at - \cos 2bt}
\]
and
\[
r_2\pm r_1=\frac{\cos \vartheta \sinh at \pm \sin \vartheta \sin bt}{\cosh at \mp \cos bt}.
\]
In the following, take into account \eqref{eq:conditionsangles}. For the second asymptotic formula, we have 
\[
  \lim_{t\to 0^+}\frac{\cosh 2at - \cos 2bt}{t^2}=4
\]
and two possibilities. If $|\omega|<\frac{\pi}{2}-|\vartheta|$,
\[
    \lim_{t\to 0^+}\frac{\cos \vartheta \sinh 2at + \sin \vartheta \sin 2bt}{t}=
     2a \cos \theta +2b \sin \theta=2\cos(\omega+\theta)>0.
\]
On the other hand, if $|\omega|=\frac{\pi}{2}-|\vartheta|$,
\[
    \lim_{t\to 0^+}\frac{\cos \vartheta \sinh 2at + \sin \vartheta \sin 2bt}{t^3}=
    \frac{2\sin(4\theta)}{4}.
\]
This yields the second asymptotic formula. For the third asymptotic formula, taking similar limits gives the following. 
If $|\omega|<\frac{\pi}{2}-|\vartheta|$,  
\[
    (r_2^2- r_1^2)^{-1}=\frac{4}{a^2\cos^2 \vartheta -b^2 \sin^2 \vartheta}+O(t^2).
\]
If $|\omega|=\frac{\pi}{2}-|\vartheta|$,
\[
   r_2^2-r_1^2=\frac{a^2b^2}{12}t^2+O(t^4).
\]
The remaining details in the proof are straightforward.
\end{proof}

For $x,y\in\mathbb{R}$,
\begin{align*}
\Re[2 w_1xy-w_2(x^2+y^2)]&=2r_1 xy-r_2(x^2+y^2)\\
&=2r_1xy+\frac{r_1^2}{r_2}x^2-\frac{r_1^2}{r_2}x^2-r_2x^2-r_2y^2 \\
&=-r_2\left[\frac{r_1}{r_2}x-y\right]^2-\frac{r_2^2-r_1^2}{r_2}x^2.
\end{align*}
If \eqref{eq:forSchattennorm} holds true, then
\begin{align*}
       \int_{x\in\mathbb{R}}\int_{y\in\mathbb{R}}|M_{\vartheta}(\tau,x,y)|^2\, \mathrm{d}y\,\mathrm{d}x&=\frac{|w_1|}{\pi}
       \int_{x\in\mathbb{R}}\int_{y\in\mathbb{R}} \e^{-2r_2\left[\frac{r_1}{r_2}x-y\right]^2} \e^{-2\frac{r_2^2-r_1^2}{r_2}x^2}\,\mathrm{d}y\,\mathrm{d}x \\
       &=\frac{|w_1|}{2\sqrt{r_2^2-r_1^2}}.
\end{align*}
Hence, by analytic continuation it follows that
\[
     \e^{-H_{\vartheta}\tau}f(x)=\int_{-\infty}^\infty
     M_{\vartheta}(\tau,x,y)f(y) \mathrm{d}y \qquad \forall \tau\in \mathcal{T}_{\vartheta}
\]
and 
\[
\|e^{-H_{\vartheta}\tau}\|_{2}^2= \frac{\pi|w_1|}{2\sqrt{r_2^2-r_1^2}}<\infty \qquad \forall \tau\in \mathcal{T}_{\vartheta}.
\]
This is the extension of Mehler's formula obtained in \cite{2017AlemanViola} for $H_\vartheta$.

The semigroup property 
\[
   \e^{-H_{\vartheta}(\tau+\sigma)}=\e^{-H_{\vartheta}\tau}\e^{-H_{\vartheta}\sigma}
\]
is valid for all $\tau,\sigma\in \mathcal{S}_{\vartheta}$.
By analytic continuation this property extends also to $\tau,\sigma\in\mathcal{T}_{\vartheta}$ such that $\tau+\sigma\in \mathcal{T}_{\vartheta}$. Hence
\[
    \e^{-H_{\vartheta}\tau}\in \mathcal{C}_1 \quad \forall
    \tau\in \mathcal{T}_\vartheta.
\]
Since $\mathcal{T}_{\vartheta}$ is open, there exists $\varepsilon>0$ such that $(1\pm\varepsilon)\tau\in \mathcal{T}_{\vartheta}$ for $\tau\in \mathcal{T}_{\vartheta}$. Then, indeed,
\[
    \|\e^{-H_{\vartheta}\tau}\|_1=\|\e^{-H_{\vartheta}(1-\varepsilon)\tau}\e^{-H_{\vartheta}(1+\varepsilon)\tau}\|_1\leq
    \|\e^{-H_{\vartheta}(1-\varepsilon)\tau}\|_2\|\e^{-H_{\vartheta}(1+\varepsilon)\tau}\|_2<\infty.
\] 
Moreover, from the asymptotic formulas  in Lemma~\ref{signandasymptoticsofrealpartswj} and the periodicity of the hyperbolic functions, it follows the next statement. Recall \eqref{eq:conditionsangles}.

\begin{Le} \label{asymptoticssemigroup} 
For all $k\in \mathbb{Z}$, $\vartheta \in\left(-\frac{\pi}{2} ,\frac{\pi}{2}\right)$ and $\omega\not =\pm\frac{\pi}{2}$
fixed,
\[
      \|e^{-H_{\vartheta}(\e^{\ci\omega}t+\ci k \pi)}\|_{2}^2=\begin{cases}
    \frac{\pi}{8(\cos^2 \vartheta-\sin^2 \omega)^{\frac12}}t^{-1}+O(1)   & |\omega|<\frac{\pi}{2}-|\vartheta|  \\
 \frac{\pi \sqrt{3}}{8\sin^2(2\vartheta)}t^{-2}+O(t^{-1}) &   |\omega|=\frac{\pi}{2}-|\vartheta|
      \end{cases}
\]
as $t\to 0^+$.
\end{Le}

\section{Perturbations of the non-selfadjoint harmonic oscillator}
\label{sec3}
We now consider locally integrable potentials $V:\mathbb{R}\longrightarrow \mathbb{C}$ satisfying \eqref{precisecondpot}.  Below we take the maximal domain
\[
    \mathrm{D}(V)=\{f\in L^2(\mathbb{R}):\int_\mathbb{R} |V(x)|^2|f(x)|^2\mathrm{d}x<\infty\}
\]
and denote with the same letter $V$ the operator of multiplication in that domain. We begin by showing that $V$ is a $\mathcal{PC}_r$ perturbation of $H_{\vartheta}$ for suitable $r>1$.

\begin{Th} \label{Th:theopert}
Let $\vartheta \in\left(-\frac{\pi}{2} ,\frac{\pi}{2}\right)$ and $\omega\not =\pm\frac{\pi}{2}$. If \eqref{precisecondpot} holds true, then $V$ is a class~$\mathcal{PC}_r$ perturbation of $\e^{\ci \omega}H_{\vartheta}$ for all 
\begin{equation} \label{eq:rangeofr}
r>\begin{cases}  \frac{2}{(2-\alpha)} & |\omega|<\frac{\pi}{2}+|\vartheta| \\
\frac{4}{(2-\alpha)} & |\omega|=\frac{\pi}{2}+|\vartheta|.  
\end{cases}
\end{equation}
\end{Th} 
\begin{proof}
In this proof the constants $k_j>0$ are independent of $n$, $\omega$ or $\vartheta$, but might depend on $p,\,r$ or $\alpha$. Assume that $|\omega|=\frac{\pi}{2}-|\vartheta|$. We include full details in this case only as the other one is very similar.

Our first goal is to construct a potential $\tilde{V}$ with the same growth as $V$ such that, for some $\varepsilon>0$,
\begin{equation} \label{eq:prooftheoper4}
   \left\|\tilde{V}\e^{-H_{\vartheta}\e^{\ci \omega}t}\right\|_{r}=O( t^{-1+\varepsilon}) \qquad \text{as }t\to 0^+.
\end{equation}
Let $n\in \mathbb{N}$. Let
\[
    \chi_n(x)=\chi_{[2^n,2^{n+1}]}(x).
\]
Then
\begin{equation}  \label{eq:prooftheoper0}
\begin{aligned}
   \left\|\chi_n\e^{-H_{\vartheta}\e^{\ci \omega}t}\right\|_2^2&=\frac{2|w_1|}{\pi}
 \int_{y\in \mathbb{R}}\e^{-2 r_2 y^2}\mathrm{d}y\int_{2^n}^{2^{n+1}}\e^{-\frac{r_2^2-r_1^2}{r_2}x^2}\mathrm{d}x \\
&=\frac{2|w_1|\sqrt{\pi}}{\sqrt{2r_2}}
 \int_{2^n}^{2^{n+1}}\e^{-\frac{r_2^2-r_1^2}{r_2}x^2}\mathrm{d}x.
\end{aligned}
\end{equation}
Let $p>0$. Then there exist $k_1>0$ such that
\[
\left\|\chi_n\e^{-H_{\vartheta}\e^{\ci \omega}t}\right\|_2^2\leq k_1 \frac{|w_1|r_2^{\frac{p}{2}}}{(r_2^2-r_1^2)^{\frac{p+1}{2}}}2^{-np}.
\]
From Lemma~\ref{signandasymptoticsofrealpartswj} it then follows that
 \begin{equation} \label{eq:prooftheoper1}
\left\|\chi_n\e^{-H_{\vartheta}\e^{\ci \omega}t}\right\|_2^2\leq k_2  
2^{-np}
t^{-\frac{p+4}{2}} \qquad \forall t\in(0,1).
\end{equation}
Using the semigroup property, then $\chi_{n}\e^{-H_{\vartheta}t}\in \mathcal{C}_2$ for all $t>0$. Also, note that
\[
    \left\|\chi_n\e^{-H_{\vartheta}\e^{\ci \omega}t}\right\|_{\infty}\leq
      \left\|\e^{-H_{\vartheta}\e^{\ci \omega}t}\right\|_{\infty}<1 \qquad \forall t>0.
\]
Then 
\[
     \left\|\chi_n\e^{-H_{\vartheta}\e^{\ci \omega}t}\right\|_{r}\leq
    k_4 2^{-\frac{np}{r}} t^{-\frac{p+4}{2r}}
\qquad \forall t\in(0,1),\,r>2.
\]

Let
\[
   \tilde{V}(x)=\sum_{n=0}^\infty 2^{\alpha(n+1)}\chi_n(x).
\]
If $r$ and $p$ are such that
\begin{equation}   \label{eq:prooftheoper2}
    \frac{p+4}{2r}<1 \qquad \text{and} \qquad \frac{p}{r}>\alpha, 
\end{equation}
then, for some $\varepsilon>0$,
\begin{equation} \label{eq:prooftheoper3}
   \left\|\tilde{V}\e^{-H_{\vartheta}\e^{\ci \omega}t}\right\|_{r}\leq
    k_5 \left(\sum_{n=0}^\infty 2^{n\left(\alpha-\frac{p}{r}  \right)}  \right)
t^{-\frac{p+4}{2r}}\leq k_6 t^{-1+\varepsilon} \qquad \forall t\in(0,1).
\end{equation}
This confirms \eqref{eq:prooftheoper4}.

Note that the condition \eqref{eq:prooftheoper2} is satisfied for $0\leq \alpha <2$, whenever $r>\frac{4}{2-\alpha}$ and $p\in (\alpha r,2r-4)$. That is precisely the requirement on $\alpha$ in the hypothesis above.  Fix $r$ and $p$ in this range. We now show that $A=\tilde{V}$ is a class~$\mathcal{PC}_r$ perturbation of $T=\e^{\ci\omega}H_{\vartheta}$. The operator $\tilde{V}\e^{-H_\vartheta \e^{i\omega}t}$ has integral kernel $\tilde{V}(x)M_{\vartheta}(\e^{\ci \omega}t,x,y)$. For all $t>0$ fixed, 
\[
    \int_{x\in \mathbb{R}}\int_{y\in \mathbb{R}} |\tilde{V}(x)M_{\vartheta}(\e^{\ci \omega}t,x,y)|^2 \mathrm{d}y \mathrm{d}x<\infty
\]
as a consequence of \eqref{eq:prooftheoper0} and the definition of $\tilde{V}$.
Then $\tilde{V}\e^{-H_\vartheta \e^{i\omega}t}\in \mathcal{C}_2$ and it is also continuous in $\mathcal{C}_2$ for all $t>0$. Hence $\tilde{V}\e^{-H_\vartheta \e^{i\omega}t}\in \mathcal{C}_q$ for all $q>2$ also and it is continuous in the norm of
$\mathcal{C}_q$. This includes $q=\infty$. Hence for all $f\in L^2(\mathbb{R})$, $\tilde{V}\e^{-H_\vartheta \e^{i\omega}t}f\in L^2(\mathbb{R})$. Thus
\[
       \operatorname{D}(\tilde{V})\supset \bigcup_{t>0}\e^{-H_\vartheta \e^{i\omega}t}(L^2(\mathbb{R})).
\]
Finally, the fact that
\[
      \int_0^1 \left\| \tilde{V}\e^{-H_\vartheta \e^{i\omega}t}\right\|_r\mathrm{d}t<\infty
\]
is guaranteed by \eqref{eq:prooftheoper3}. 

In order to complete the proof for  $|\omega|=\frac{\pi}{2}-|\vartheta|$, note that $\tilde{V}$ is invertible and that a generic $V$ satisfying \eqref{precisecondpot} with the same $\alpha$ is such that $\mathrm{D}(\tilde{V})=\mathrm{D}(V)$. Therefore Lemma~\ref{Th:perofclassPCqpert} ensures that $V$ is also a class~$\mathcal{PC}_r$ perturbation for $r$ in the stated range.

Our only additional comment about the case $|\omega|<\frac{\pi}{2}-|\vartheta|$ is that the exponent of $t$ in \eqref{eq:prooftheoper1} changes to $\frac{p+2}{2}$. This leads to replacing the left of \eqref{eq:prooftheoper2} by $\frac{p+2}{2r}$ and this yields $r>\frac{2}{2-\alpha}$.
\end{proof}

\begin{Rm} \label{remark:dissipativityofthepert}
Here the potential $V$ can be accretive or otherwise. For example $V(x)=\e^{ix}|x|^\alpha$ where $0<\alpha<2$ is included in this theorem.
\end{Rm}

By combining the above with Theorem~\ref{Th:Gibbsinsector} it follows that, for $\vartheta\not=0$, the non-selfadjoint Schr\"odinger operator $H_\vartheta+V$ is the generator of a Gibbs semigroup $\e^{-(H_\vartheta+V)\tau}\in \mathcal{C}_1$ for all $\tau\in \overline{\mathcal{S}_{\vartheta}}\setminus\{0\}$ holomorphic in the maximal sector $\mathcal{S}_{\vartheta}$. Moreover, for $|\omega| \leq \frac{\pi}{2}-|\vartheta|$ and $r$ in the range determined by \eqref{eq:rangeofr}, 
\[
\|\e^{-(H_{\vartheta}+V)\e^{\ci \omega} t}\|_r \sim \|\e^{-H_{\vartheta}\e^{\ci \omega} t}\|_r
\qquad t\to 0^+.
\]
If $|\omega|<\frac{\pi}{2}-|\vartheta|$, this range includes $r<2$ and we get from  Lemma~\ref{asymptoticssemigroup} that
\[
    \|\e^{-(H_{\vartheta}+V)\e^{\ci \omega} t}\|_2\sim t^{-\frac{1}{2}} \qquad t\to 0^+.
\]
From the interpolation inequality in $\mathcal{C}_r$ and the fact that $\e^{-H_{\vartheta} \e^{\ci \omega} t}$ are contraction semigroups, it follows that for $r>2$,
\[
\|\e^{-H_{\vartheta}\e^{\ci \omega} t}\|_r\leq \|\e^{-H_{\vartheta}\e^{\ci \omega} t}\|_2^{\frac{2}{r}} \|\e^{-H_{\vartheta}\e^{\ci \omega} t}\|_{\infty}^{1-\frac{2}{r}}\sim \begin{cases}
  t^{-\frac{1}{r}} &  |\omega| < \frac{\pi}{2}-|\vartheta| \\
  t^{-\frac{2}{r}} & |\omega| = \frac{\pi}{2}-|\vartheta|. \end{cases}
\]
Then,
\begin{equation} \label{eq:asymptrace}
\|\e^{-(H_{\vartheta}+V)\e^{\ci \omega} t}\|_r= \begin{cases}
  O(t^{-\frac{1}{r}}) &  |\omega| < \frac{\pi}{2}-|\vartheta| \\
  O(t^{-\frac{2}{r}}) & |\omega| = \frac{\pi}{2}-|\vartheta| \end{cases}
\end{equation}
as $t\to 0^+$. As we shall see next, combining this with \eqref{eq:spectheo} leads to asymptotics for the eigenvalues of the perturbed operator.

Let
\[
\operatorname{Spec} (H_{\vartheta}+V)=\{\lambda_n\}_{n=1}^\infty. 
\]
Denote
\begin{equation*} \label{eq:thebetan}
  \begin{aligned} \alpha_n&=\Re(\lambda_n), \\
    \beta_n^{\pm}&=\Re\left(\e^{\pm \ci \left(\frac{\pi}{2}-|\vartheta|\right)}\lambda_n\right)=\Re\left(\e^{\ci \omega}\lambda_n\right) & \text{ for } |\omega|=\frac{\pi}{2}-|\vartheta|.
\end{aligned}    
\end{equation*}
Since $\e^{-(H_\vartheta+V)\e^{\ci \omega}t}$ is compact for all $t>0$, it follows that $\alpha_n,\,\beta^{\pm}_n\to\infty$,  \emph{cf.} \cite[Theorem~8.2.13]{2007Davies}. 

\begin{Co} \label{Co:naturespectrumetal}
Let $V$ satisfy \eqref{precisecondpot}. Then, the resolvent $(H_\vartheta+V-z)^{-1}\in \mathcal{C}_q$ for all $q>1$. Moreover, $H_\vartheta+V$ has an infinite number of distinct eigenvalues and a complete set of root vectors\footnote{We follows the standard terminology here, meaning that the set of finite linear combinations of all the root vectors has zero as orthogonal complement.}. 
\end{Co}
\begin{proof}
By adding to $H_\vartheta+V$ a sufficiently large constant, without loss of generality we can assume that $\varphi(H_\vartheta+V)=-1$ and take $z=0$. The  inverse Laplace transform identity \eqref{eq:Laplacetransform} for $T=H_\vartheta+V$ gives
\begin{equation}
\label{eq:LaplaceTransformperturbedharmonic}
    (H_\vartheta+V)^{-1}f=\int_{0}^\infty \e^{-(H_\vartheta+V)t}f \mathrm{d}t \qquad \forall f\in L^2(\mathbb{R}).
\end{equation}
From \eqref{eq:asymptrace} and the assumption on the uniform growth bound, it follows that
\begin{align*}
\int_{0}^\infty &\|\e^{-(H_\vartheta+V)t}\|_q\mathrm{d}t  \\
&\leq \int_{0}^2 \|\e^{-(H_\vartheta+V)t}\|_q\mathrm{d}t+\int_{2}^\infty \|\e^{-(H_\vartheta+V)}\|_q\|\e^{-(H_\vartheta+V)(t-1)}\|_{\infty}\mathrm{d}t \\
&\leq k_6\int_0^2 t^{-\frac{1}{q}} \mathrm{d}t +k_7 \int_2^\infty \e^{-t} \mathrm{d}t<\infty. 
\end{align*}
Then, the integral in \eqref{eq:LaplaceTransformperturbedharmonic} is absolutely convergent in $\|\cdot\|_q$ and so the associated operator belongs to $\mathcal{C}_q$.

The second and last statements are classical. A concyse proof is achieved by means of a direct application of \emph{e.g.} \cite[Corollary~4.10]{2016Shkalikov}. Indeed, taking $V=0$ in the first statement just shown, yields that $H_\vartheta$ has ``order'', in the sense of \emph{loc. cit.} p.918, any constant less than one. We know that $H_\vartheta$ is m-sectorial with angle $\vartheta=\frac{\gamma \pi}{2}$ for $\gamma<1$ and $V$ is $H_\vartheta$-bounded with bound zero. That is ``completely subordinate'' in the terminology of  \emph{loc. cit.} p.910, so the hypotheses of the mention corollary are satisfied.
\end{proof}

As we shall see next, lower bounds on the asymptotic behaviour of $\alpha_n$ and $\beta_{n}^{\pm}$ can be derived from Lidskii's inequality.

\begin{Co} \label{Co:asymptoticeigenvalues}
Assume that $\vartheta\not=0$. Let $V$ satisfy \eqref{precisecondpot}. Then
there exist constants $K>0$ and $n_0\in \mathbb{N}$ such that \[
       \alpha_n\geq Kn \quad\text{and} \quad \beta_n^{\pm}\geq Kn^{\frac12} \qquad \forall n\geq n_0.
\]
\end{Co}
\begin{proof}  
  Recall that \[ \operatorname{Spec}(\e^{-(H_{\vartheta}+V)\e^{\ci \omega}t})=\{ \e^{-\e^{\ci \omega}\lambda_k t}  \}_{k=1}^\infty\cup \{0\}. \]
  Let $r>\frac{4}{2-\alpha}$.
For $t=rs$, 
\[
    \sum_{k=1}^\infty|\e^{-\e^{\ci \omega}\lambda_k s} |^r \leq \|\e^{-(H_{\vartheta}+V)\e^{\ci \omega}s} \|_{r}^{r}=O(s^{-2}) \qquad s\to 0^+. 
\]
Then 
\[
     \sum_{k=1}^\infty \e^{-\beta^{\pm}_k t} = \sum_{k=1}^\infty
     (\e^{-\beta_k^{\pm}s})^r =O( s^{-2})=O( t^{-2}) \qquad t\to 0^+.
\]
Assume that the eigenvalues are ordered so that $\beta_n^{\pm}$ is non-decreasing (with possibly different orders for the two cases $\pm$). Then
\[
    n\e^{-\beta_n^\pm t}=\sum_{k=1}^n\e^{-\beta_n^\pm t}\leq  \sum_{k=1}^\infty \e^{-\beta^{\pm}_k t}.
\]
Hence there exist a constant $k_8>0$ such that
\[
    n\e^{-\beta^{\pm}_nt}\leq \frac{k_8}{t^2} \qquad \forall 0<t\leq t_0
    \]
    where $t_0>0$ is small enough, this for all $n\in \mathbb{N}$. Take $n_0\in \mathbb{N}$ large enough such that $\frac{1}{\beta^{\pm}_{n_0}}<t_0$. Putting $t=\frac{1}{\beta^{\pm}_n}$, gives \[
    n\e^{-1}\leq k_8(\beta_n^{\pm})^2 \qquad \qquad \forall n\geq n_0.     \]
This ensures the validity of the claim for $\beta_n^\pm$.  The conclusion for the case of $\alpha_n$ is achieved with a similar argument noting that the asymptotic changes to  $n\e^{-\alpha_nt}=O(t^{-1})$ for $t\to 0^+$.
\end{proof}

The estimate above is optimal for $\alpha_n$, as it should hold true for $V=0$. Since
\[
    \sum_{n=1}^\infty \e^{-n^{1/2}t}\geq \int_1^\infty \e^{-x^{1/2} t}\mathrm{d}x\sim t^{-2} \qquad t\to 0^+,
    \]
we know that the exponent $\frac12$ for $\beta_n^{\pm}$ above is also optimal, given the asymptotic behaviour of the $\mathcal{C}_r$ norm of the semigroup. However, it is not clear that the exponent in the latter is optimal for potentials satisfying \eqref{precisecondpot}. That is, we do not know if the exponent for $t$ in the formula \eqref{eq:asymptrace} is optimal.

From general principles, it follows that the $\varepsilon$-pseudospectrum
\[
\operatorname{Spec}_{\varepsilon}(H_{\vartheta})\subset \{z+s\e^{\ci\omega}:z\in \operatorname{Num}(H_{\vartheta}) ,\,0\leq s\leq \varepsilon,\, |\omega|\leq \pi\}  
\] for all $\varepsilon>0$. In fact, $\operatorname{Spec}_{\varepsilon}(H_{\vartheta})$ is known to obey the following more precise enclosures for fixed $1< q_1\leq 3<q_2<\infty$. Write 
\[
    R_q=\{r+r^q \e^{\ci \omega}:r\geq 0,\, |\omega|\leq |\vartheta|\}.
\]
For all $\varepsilon_1>0$ there exists $E_1>0$ such that
\[
    (E_1+R_{q_1}) \subset \operatorname{Spec}_{\varepsilon_1}(H_{\vartheta}),
\]
see \cite{2001Boulton}. But for all $\gamma,\, E_2>0$ there exist $\varepsilon_2>0$ such that  
\[
\operatorname{Spec}_{\varepsilon}(H_{\vartheta}) \subset
(E_2+R_{q_2}) \cup \bigcup_{n=1}^\infty \{z\in\mathbb{C}:|2n+1-z|\leq \gamma\} \qquad \forall \varepsilon\leq \varepsilon_2,
\]
see \cite{2006Pravda-Starov}. See also \cite{2017KrejcirikSiegl,2004Denckeretal, 1999Zworski}. Then, according to Corollary~\ref{Co:asymptoticeigenvalues}, asymptotically the eigenvalues of $H_\vartheta+V$ lie way inside $\operatorname{Spec}_\varepsilon(H_\vartheta)$ and the distance from $\partial\operatorname{Spec}_\varepsilon(H_\vartheta)$ to $\lambda_n$ grows (at least like $n^{1/2}$) as $n\to \infty$.

The following result gives an indication of the shape of the pseudospectra of $H_\vartheta+V$. It implies that the distance from the real axis to $z\in\partial \operatorname{Spec}_{\varepsilon}(H_\vartheta+V)$ is $o(\Re(z))$ as $z\to \infty$.

\begin{Co} \label{Co:asymptoticperturbedHO} 
Assume that $\vartheta\not=0$. Let $V$ satisfy \eqref{precisecondpot}. Then 
\[
    \lim_{\rho\to \infty} \| (H_\vartheta+V -
    \e^{\pm \ci \vartheta} \rho-\beta)^{-1}\|=0 
    \]
 for all $\beta\in \mathbb{R}$.   
\end{Co}
\begin{proof}
  By rotating the operator and directly applying Theorem~\ref{Th:resolventnormatinfinity}, the conclusion follows for all $\beta$ sufficiently negative. We use Corollary~\ref{Co:asymptoticeigenvalues} and a spectral decomposition similar to that in \cite[\S2.2]{1980Davies} to show the property for all $\beta\in \mathbb{R}$. 

Fix $\gamma\in \mathbb{R}$. By virtue of Corollary~\ref{Co:asymptoticeigenvalues}, there exists $N\in \mathbb{N}$ such that
  \[
      \{\lambda_n\}_{n=N+1}^\infty \subset \mathcal{S}(-|\vartheta|,|\vartheta|)+\gamma. 
  \]
Let $C$ be a simple Jordan curve such that only $\{\lambda_n\}_{n=1}^N$ are in its interior. Let
  \[
     P=\frac{1}{2\pi \ci} \int_{C}(z-H_\vartheta-V)^{-1} \mathrm{d}z
     \]
     be the corresponding Riesz projector. Let
     \[
 \mathcal{M}_1=P[L^2(\mathbb{R})]\subset \mathrm{D}(H_\vartheta) \quad \text{and} \quad \mathcal{M}_2=(I-P)[L^2(\mathbb{R})].  
\]
The subspace $\mathcal{M}_1\subset \mathrm{D}(H_\vartheta)$ is finite-dimensional, $L^2(\mathbb{R})=\mathcal{M}_1 + \mathcal{M}_2$ and $\mathcal{M}_1 \cap \mathcal{M}_2=\{0\}$. Generally there is no orthogonality between $\mathcal{M}_1$ and $\mathcal{M}_2$.

Let
\[
  [H_\vartheta+V]_j=(H_\vartheta+V)\upharpoonright_{\mathcal{M}_j}:\mathcal{M}_j\cap \mathrm{D}(H_\vartheta) \longrightarrow \mathcal{M}_j
    \]
    denote the corresponding restriction operators. Then \[
  \operatorname{Spec}([H_\vartheta+V]_1)=\{\lambda_n\}_{n=1}^N \quad \text{and}\quad \operatorname{Spec}([H_\vartheta+V]_2)=\{\lambda_n\}_{n=N+1}^\infty.
  \]
  Since $\mathcal{M}_j$ are invariant subspaces for the resolvent, then they are also invariant under the action of the $C_0$-semigroup (this is guaranteed from the fact that the latter commutes with the resolvent). The restriction operators are generators of the corresponding Gibbs semigrous on the subspaces, that is \cite[Theorem~2.20]{1980Davies}\[\e^{-[H_\vartheta +V]_j\tau}=\e^{-(H_\vartheta +V)\tau}\upharpoonright_{\mathcal{M}_j}:\mathcal{M}_j\longrightarrow \mathcal{M}_j\] for all $\tau\in \overline{\mathcal{S}_\vartheta}\setminus\{0\}$ holomorphic in $\mathcal{S}_\vartheta$. Since the restriction to $\mathcal{M}_1$ is bounded, we have \[\lim_{|z|\to\infty} \|([H_\vartheta+V]_1-z)^{-1}\|=0.\] According to Theorem~\ref{Th:resolventnormatinfinity} applied to $T=\e^{\pm\ci\left(\frac{\pi}{2}-|\vartheta|\right)}[H_\vartheta+V]_2$, we have \[\lim_{\rho \to \infty} \|([H_\vartheta+V]_2-\e^{\pm \vartheta} \rho-\beta)^{-1}\|=0 \qquad \forall \beta<\gamma.\] Since \[ \|(H_\vartheta+V-z)^{-1}\|\leq \|([H_\vartheta+V]_1-z)^{-1}\|+\|([H_\vartheta+V]_2-z)^{-1}\|,\] the conclusion indeed follows for all $\beta<\gamma$. We complete the proof by choosing $\gamma$ arbitrarily large.  
\end{proof}

\bibliographystyle{plain}

\end{document}